\documentclass[11pt]{amsart}
\usepackage[all]{xy}
\usepackage{amssymb}
\usepackage{enumitem}
\usepackage{amsthm}
\usepackage{colortbl}
\usepackage[dvipsnames]{xcolor}
\usepackage{hyperref}
\hypersetup{colorlinks=true,linkcolor=blue,citecolor=magenta}
\usepackage{amsmath}
\usepackage{amscd,enumitem}
\usepackage{verbatim}
\usepackage{eurosym}
\usepackage{float}
\usepackage{cleveref}
\usepackage{color}
\usepackage{rotating}
\usepackage{dcolumn}
\usepackage[mathscr]{eucal}
\usepackage[all]{xy}
\usepackage{hyperref}
\usepackage{bbm}
\usepackage{makecell}
\usepackage{wasysym}
\usepackage[margin=1 in]{geometry} 
\usepackage{tikz}
\usepackage{graphicx}
\usepackage{mathtools}

\newcommand\myhdots{\hbox to 1.5em{.\hss.\hss.}}
\usetikzlibrary{automata,arrows,positioning,calc}
\DeclareFontFamily{U}{wncy}{}
    \DeclareFontShape{U}{wncy}{m}{n}{<->wncyr10}{}
    \DeclareSymbolFont{mcy}{U}{wncy}{m}{n}
    \DeclareMathSymbol{\Sh}{\mathord}{mcy}{"58} 
\newtheorem*{thm*}{Theorem}
\newtheorem*{conj*}{Conjecture}

\newtheorem{theorem}{Theorem}[section]
\newtheorem{cor}[theorem]{Corollary}

\newtheorem{lemma}[theorem]{Lemma}

\theoremstyle{definition}
\newtheorem{definition}[theorem]{Definition}

\title{An Asymptotic for Non-Wieferich Numbers}
\author{Casia Siegel}

\begin{document}


\begin{abstract} 
 A cyclic number is a positive integer $n$ such that $\gcd(n, \varphi(n)) = 1$. Erdős showed that the number of cyclic numbers up to $x$, $C(x)$, has $C(x) \sim e^{-\gamma}x/\log_3(x)$ asymptotically as $x$ tends to infinity. We define Wieferich numbers to be odd, positive integers $n$ such that $(n, (2^{\ell(n)}-1)/n) > 1$, where $\ell(n)$ is the smallest positive integer with $2^{\ell(n)} \equiv 1 \mod n$. Franco and Pomerance showed that non-Wieferich numbers have density zero among the odd integers. We show that the counting function of $n \leq x$ in the symmetric difference of non-Wieferich numbers and cyclic numbers is $O(x/(\log_3(x))^2)$, under a mild assumption on the counting function of Wieferich primes.

\end{abstract}

\maketitle 

\section{Introduction}

If one defines $f(n)$ to be the largest odd divisor of $3n+1$, the famous and unproven Collatz conjecture asserts that for $n\in \mathbb{N}$ the sequence $n, f(n), f(f(n)), \dots$ will eventually reach the number 1. However, if one were to change $3n+1$ to $qn+1$ for a positive odd integer $q$, this does not necessarily happen. For example, if we let $q=5$ not all sequences will end up at 1 (such as $n=13$ which has the repeating sequence $13,33,83,13,33,\dots$). \textit{Crandall numbers} are $q$ for which there exists an $n$ where this sequence does not terminate at 1. In \cite{Pommerance}, Franco and Pomerance examine the distribution of Crandall numbers. As an intermediary, they first consider Wieferich numbers.

\begin{definition}
    A positive odd integer $q$ is \textit{Wieferich} if $\gcd\bigg(\frac{2^{\ell(q)}-1}{q},q\bigg) >1$, where $\ell(q)$ is the smallest positive integer such that $2^{\ell(q)} \equiv 1 \mod q$. 
\end{definition}

This definition comes from old work on Fermat's last theorem in which Wieferich showed that the first case of Fermat's last theorem is true for prime exponent $p$ as long as $p$ is not a Wieferich prime. The distribution of Wieferich primes has long been of interest, but it is still not known whether there are finitely or infinitely many Wieferich primes, with 1093 and 3511 being the only two known. 

It is also not known whether there are infinitely many non-Wieferich primes; however, Silverman \cite{Silverman} showed that $\#\{n\leq x : n \text{ is non-Wieferich}\}\gg \log(x)$ under the assumption of the $abc$-conjecture. It is conjectured that most primes are non-Wieferich. In fact, in \cite[Section 3]{search}, Crandall, Dilcher and Pomerance record their expectation that
$$W(x) := \#\{p \leq x : \text{ $p$ is a Wieferich prime}\} \approx \log_2(x)$$
\noindent where we use the notation $\log_k(x)$ to mean the $k$-th iterate of $\log(x)$.  
For the sake of this paper, we will only be assuming the much weaker statement that
\begin{equation}\label{eq:assume}
    W(x)  = O(x/\log^\alpha(x))
\end{equation}

\noindent for some $\alpha > 2$.

In \cite{Pommerance} it is shown that all Wieferich numbers are Crandall numbers, and that the Wieferich numbers have density 1 in the odd integers. In this paper, we will examine their complement, the non-Wieferich numbers for which $\gcd\bigg(\frac{2^{\ell(n)}-1}{n},n\bigg) =1$. The goal of this paper is to find an asymptotic formula for the count of non-Wieferich numbers $n\leq x$, as $x$ tends to infinity. We start with the following definition.

\begin{definition}  Let $N = \{n \in \mathbb{N}: n \text{ odd and non-Wieferich}\}$, and $N(x) = \# \{n \leq x: n \in N \}$.
    
\end{definition}

Any integer $n$ such that $\gcd(n,\varphi(n))=1$, where $\varphi(n)$ is Euler's totient function, is called  a \textit{cyclic number}. The name first comes from the result, first proved by \cite{german}, that $n$ is cyclic if and only if every group of order $n$ is cyclic.

\begin{definition} Let $C = \{n \in \mathbb{N} :n \text{ is cyclic}\}$, and $C(x) = \#\{ n \leq x : n \in C\}$.
\end{definition}

Erdős showed in \cite{erdosss} (see also the proof of this in \cite{Vaughan} on pages 387-389) that 
\begin{equation}\label{eq:cyclic}
C(x) \sim \frac{e^{-\gamma}x}{\log_3(x)}.
\end{equation}
\noindent where $\gamma \approx 0.577$ is the \textit{Euler-Mascheroni constant}. Furthering this in \cite{Pollack2}, Pollack obtained an asymptotic series expansion for $C(x)$ in descending powers of $\log_3(x)$.

\begin{definition}  We let $S$ be the symmetric difference between $C$ and $N$, i.e. $S= \{n \in N \Delta C \}$.
\end{definition}

We will prove, under the assumption of $\eqref{eq:assume}$, that $N$ and $C$ are essentially the same: each contains 100\% of the elements of the other, asymptotically. More precisely, we have the following.

\begin{theorem}\label{main} Assume \eqref{eq:assume}. Let $S(x) = \#\{n \leq x : n \in S\}$. Then
$$S(x) = O\bigg(\frac{x}{(\log_3(x))^2}\bigg)$$
\noindent as $x\rightarrow \infty.$

\end{theorem}

The following is an immediate consequence of \eqref{eq:cyclic} and Theorem \ref{main}.

\begin{cor}\label{asyms} Assume \eqref{eq:assume}. For $N(x) = \#\{n \leq x : n \in N\}$,
$$N(x) \sim C(x) \sim \frac{e^{-\gamma}x}{\log_3(x)}$$
\noindent as $x \rightarrow \infty$.
\end{cor}

A close cousin of Corollary \ref{asyms} is proved in \cite{Luca}. We let $\lambda(n)$ be the smallest positive integer such that $a^{\lambda(n)} \equiv 1 \mod n$ for all $a$. Then define an odd number $n$ as \textit{$\lambda$-Wieferich} whenever gcd$((2^{\lambda(n)}-1)/n,n)>1$, and let $U = \{ n \in \mathbb{N} : n \text{ is not $\lambda$-Wieferich}\}$. Then \cite[Theorem 10]{Luca} states that
\begin{equation}\label{eq:luca}
U(x) :=\#\{n\leq x : n \in U\} \ll \frac{x}{\log_3(x)}.
\end{equation}
As $\ell(n) \mid \lambda(n)$, for any $n \in U$ we see that
$$\gcd\bigg(\frac{(2^{\lambda(n)}-1)}{n},n\bigg) \mid \gcd\bigg(\frac{(2^{\ell(n)}-1)}{n},n\bigg) = 1,$$
\noindent so also $n\in N$. Therefore $U \subseteq N$. Thus, \eqref{eq:luca} agrees with Corollary \ref{asyms}. We improve on this result, under the assumption of \eqref{eq:assume}, by showing the following.

\begin{theorem}\label{improved} As $x \rightarrow \infty$,
    $$U(x)  = N(x) + O\bigg( \frac{x}{(\log_3(x))^2} \bigg),$$
which becomes
$$U(x) \sim \frac{e^{-\gamma}x}{\log_3(x)}$$
under the assumption of \eqref{eq:assume}.
\end{theorem}

In fact, we get the upper bound from these asymptotics unconditionally, that is,
\begin{equation*}
  U(x) \leq N(x) \leq \frac{e^{-\gamma}x}{\log_3(x)} + O\bigg( \frac{x}{(\log_3(x))^2}\bigg)
\end{equation*}
\noindent (see \eqref{eq:Nupper} and \eqref{eq:Uupper}).
\\

\noindent \textbf{Acknowledgments.} I thank my advisor Paul Pollack for the many discussions and conversations regarding this paper.

\section{Proofs}

We start with a sufficient condition for a number to be non-Wieferich.

\begin{lemma}\label{nonwieferich condition} Let $n$ be odd and squarefree. Then $n$ is non-Wieferich if all primes dividing $n$ are non-Wieferich and for any two primes $p,q \mid n$ we have that $p \not\equiv 1 \mod q$.
    
\end{lemma}

\begin{proof}

Assume the stated conditions on $n$, let $p\mid n$ and write $P = n/p$, where $P$ is a product of distinct primes. It suffices to show that $(p, (2^{\ell(n)}-1)/n) =1$. As $p \parallel n$, this is equivalent to showing that $p \parallel 2^{\ell(n)}-1$. We see that $\ell(n) \mid \ell(p)\ell(P) \mid \ell(p) \varphi(P)$, and so $2^{\ell(n)} -1 \mid 2^{\ell(p)\varphi(P)}-1$. As $p$ is non-Wieferich, we have that $p \parallel 2^{\ell(p)}-1$, so for some $j \not\equiv 0 \mod p$ we have that
$$2^{\ell(p)} \equiv 1 + jp \mod p^2.$$
\noindent By applying the binomial theorem,
$$2^{\ell(p)\varphi(P)} \equiv 1 + \varphi(P)jp \mod p^2.$$
\noindent But as $p \not\equiv 1 \mod q$ for any $q \mid n$ we know that $p \nmid \varphi(P)$. Thus, $p \parallel 2^{\ell(p)\varphi(P)}-1$, and so $p \parallel 2^{\ell(n)}-1$. \qedhere

\end{proof}

\begin{lemma}\label{baby} Odd, squarefree, non-Wieferich integers have no Wieferich prime factors.

\end{lemma}

\begin{proof} Assume that $n$ is squarefree such that a Wieferich prime $p$ divides $n$. As $\gcd\bigg(\frac{2^{\ell(p)}-1}{p}, p \bigg) \neq 1$, we have $p^2|2^{\ell(p)}-1$. Writing $n=kp$, where $k$ is coprime to $p$, we know
$$\frac{2^{\ell(n)}-1}{n} = \frac{(2^{\ell(p)}-1)m}{pk}$$
\noindent for some $m$. As $p^2$ divides the numerator, and $p \nmid k$, we see that $p\mid \frac{2^{\ell(n)}-1}{n} $. Thus $n$ is Wieferich.
    
\end{proof}
\begin{lemma}[Pollack]\label{paul} For each $x \geq 3$ and each positive integer $d$,
$$\sum_{\substack{\operatorname{prime }q \leq x \\ d \mid \ell(q)}}\frac{1}{q} = \beta(d) \log_2(x) + O\bigg( \frac{\log(3d)}{\varphi(d)} \bigg)$$
\noindent where $\beta(d) \asymp 1/d$. Here all implied constants are absolute.
    
\end{lemma}

 This lemma appears as \cite[Lemma 6.4]{Paul}. The precise value of $\beta(d)$ is described in  \cite[Theorem 2]{Moree}.

We now show that the number of non-Wieferich odd $n \leq x$ which are not squarefree is small enough to allow us to focus on counting squarefree $n$ for the remainder of the paper.
 
\begin{lemma}\label{notsquarefree} Define $R(x) = \#\{ n \leq x:n \text{ odd, non-squarefree, and non-Wieferich} \}$. Then
\begin{equation}
R(x) \ll \frac{x}{\log_3(x)^{5/2}}.    
\end{equation}
 
\end{lemma}

\begin{proof}

Note that $n$ is Wieferich if it can be written in the following form: $n=p^krm$, where $k>2$, $p$ and $r$ are prime, $p^k \parallel n$, and $p^k \mid \ell(r)$. This is because $\ell(p^{k+1}) \mid p^k\ell(p)\mid \ell(n)$, and so $p^{k+1} \mid 2^{\ell(n)}-1$, giving $p \mid \gcd\bigg(n, \frac{2^{\ell(n)}-1}{n}\bigg)$. Writing the largest squarefree divisor of $n$ as $s(n)$, we define
$$R_{p,k}(x) := \#\{ n \leq x:n \text{ odd, non-Wieferich}, s(n) \leq (\log_3(x))^5, \text{ and }  p^k\parallel n\}.$$
\noindent Since the number of square-full numbers up to $x$ is $O(\sqrt{x})$, one has that
$$\#\{n \leq x : s(n) > y\} \leq \sum_{\substack{m>y \\ m \text{ squarefree}}} \frac{x}{m} \ll \frac{x}{\sqrt{y}}.$$
\noindent Then, as $n$ is odd, whenever $p^k \mid n$ we have $ k \leq \log(n)/\log(p) < \log(n)$. Hence,
$$R(x) \leq \sum_{k=2}^{\lfloor\log(x)\rfloor} \sum_{p \text{ prime}} R_{p,k}(x) + O\bigg(\frac{x}{(\log_3(x))^{5/2}} \bigg).$$

For each $k$, an $n$ counted by $R_{p,k}(x)$ can be written as $n=p^km$ where $m \leq x/p^k$ and no prime $r \mid m$ has $p^k \mid \ell(r)$. Define $P_{p,k} = \{r \leq \frac{x}{p^k}: r \text{ prime, and } p^k  \mid \ell(r) \}$, and $\mathcal{P}_{p,k} = \prod_{r\in P_{p,k}}r$. Hence, by the Brun's sieve, 
\begin{align*}
R_{p,k}(x) & \leq \#\{ m \leq \frac{x}{p^k} : \gcd(m,\mathcal{P}_{p,k})=1 \} \\
 &\ll \frac{x}{p^k}\prod_{r \in P_{p,k}}\bigg(1-\frac{1}{r}\bigg) \\
  & \ll \frac{x}{p^k} \exp \bigg( \sum_{\substack{r \in P_{p,k} \\p^k \mid \ell(r)}} \frac{-1}{r} \bigg) \\
 &\ll \frac{x}{p^k}\log\bigg(\frac{x}{p^k}\bigg)^{-c/p^k}, \\
\end{align*}
\noindent where we use Lemma \ref{paul} to estimate the product on $r$, and $c$ is a positive absolute constant. Thus, summing over $p$ and $k$, we have 
\begin{align*}
\sum_{k=2}^{\lfloor\log(x)\rfloor} \sum_{p \text{ prime}} R_{p,k}(x) & \ll  x \sum_{k=2}^{\lfloor\log(x)\rfloor}\sum_{p \text{ prime}} \frac{1}{p^k}(\log(x))^{-c/p^k} \\ & \ll x \sum_{k=2}^{\lfloor\log(x)\rfloor}\sum_{p \text{ prime}}\frac{1}{p^k}\exp\bigg(\frac{-c\log_2(x)}{(\log_3(x))^5}\bigg) \\
& \ll x\exp\bigg(\frac{-c\log_2(x)}{(\log_3(x))^5}\bigg) \ll \frac{x}{\log_2(x)} .
\end{align*}
This is certainly $O\bigg(\frac{x}{(\log_3(x))^{5/2}}\bigg)$, and so we get our result. \qedhere

\end{proof}

We shall consider the two pieces of the symmetric difference, $S = (C\setminus N)\cup (N \setminus C)$, separately.

\begin{lemma}[Franco-Pomerance]\label{wie}  Let $n$ be an odd positive integer. If there exists a prime $p$ such that $p \mid  \mid n$ and $p \mid \ell(n)$, then $n$ is Wieferich.

\end{lemma}

For the proof of this, see the proof of Theorem 4 in \cite{Pommerance}.

\begin{lemma}\label{nc}  Let $(N \setminus C)(x) := \#\{ n \leq x : n \text{ is odd, non-Wieferich, and not cyclic} \}$. We have that
$$(N\setminus C)(x) = O\bigg( \frac{x}{(\log_3(x))^2} \bigg).$$  
\end{lemma}

\begin{proof}
    Let $n \in (N\setminus C)$, with $n \leq x$. For the remainder of the proof we will assume that $n$ is squarefree, as we have a suitable bound on the non-squarefree contribution from Lemma \ref{notsquarefree}. As $n$ is not cyclic, there must be some prime $p \mid n$ such that $p \mid \varphi(n)$. Thus, as $n$ is squarefree, there must be some prime $q \mid n$ such that $p \mid q-1$. We will split such $n$ into two sets according to whether or not $n$ has a prime factor $p$ exceeding $z$. For the remainder of this paper, we will set
    \begin{equation}\label{eq:z}
    z := \log_2(x).    
    \end{equation}
    We consider first $n$ with a prime factor $p\leq z$. Let $q$ be the smallest prime dividing $n$, and write $n = mq$. By Lemma \ref{wie} we have that $m$ is not divisible by any prime $r$ such that $q \mid \ell(r)$. Furthermore, by choice of $q$, we have that $m$ is not divisible by any prime $r <q$. Hence, by Brun's sieve, the number of such $n \leq x$ is
    \begin{align*}
    \frac{x}{q} \prod_{r < q} \bigg( 1 - \frac{1}{r} \bigg) \prod_{\substack{r \leq x/q \\ q \mid \ell(r)}}    \bigg( 1 - \frac{1}{r} \bigg)  & \ll   \frac{x}{q\log(q)}\prod_{\substack{r \leq x/q \\ q \mid \ell(r)}}    \bigg( 1 - \frac{1}{r} \bigg) \\
    & \ll  \frac{x}{q\log(q)} \exp \left( \sum_{\substack{r \leq x/q \\ q \mid \ell(r)}} -\frac{1}{r} \right) \\
    & \ll  \frac{x}{q\log(q)} \bigg( \log(x/q) \bigg)^{-c/q}.
    \end{align*}
\noindent Here we again use Lemma \ref{paul} to estimate the final product on $r$. Now summing over all primes $q \leq z$, and noting that $\log(x/q) \geq \frac{1}{2}\log(x)$, we get that the count of all such $n$ is
$$ \ll \sum_{\text{prime }q \leq z} \frac{x}{q \log(q)} (\log(x))^{-c/q} = \sum_{\text{prime } q\leq z} \frac{x}{q\log(q)}e^{-cz/q}  .$$

We shall split these $q$ into dyadic intervals. Define $I_0 = [2, \frac{cz}{\log_3(x)}]$, and let $I_k = [\frac{2^{k-1}cz}{\log_3(x)}, \frac{2^kcz}{\log_3(x)}]$ for $1 \leq k \leq S$, where $S$ is the smallest integer such that 
$$z < \frac{2^{S+1}cz}{\log_3(x)}.$$ 
\noindent Then $2^{S}c \leq \log_3(x)$, and so $2^S \ll \log_3(x)$. First, note that
\begin{align*}
\sum_{q \in I_0} \frac{x}{q \log(q)} e^{-cz/q} & \leq \frac{x}{\log_2(x)}\sum_{q \in I_0} \frac{1}{q\log(q)} \ll \frac{x}{\log_2(x)},
\end{align*}

\noindent where the sum over $q$ would converge over all primes, so is certainly bounded for $q \in I_0$. For any $k \geq 1$, 
\begin{align*}
\sum_{q \in I_k} \frac{x}{q \log(q)} e^{-cz/q} &\ll x\exp\bigg(\frac{-\log_3(x)}{2^{k+1}}\bigg)\sum_{q \in I_k} \frac{1}{q\log(q)} \\
& \ll \frac{x}{\log_3(x)}\exp\bigg(\frac{-\log_3(x)}{2^{k+1}}\bigg)\sum_{q \in I_k}  \frac{1}{q}.
\end{align*}

Note that for any interval $[T,2T]$, $\sum_{p \in [T,2T]} \frac{1}{p} \leq \frac{1}{T}\pi(2T) \ll \frac{1}{\log(T)}.$ Thus, we see that

$$\sum_{q \in I_k} \frac{x}{q \log(q)} e^{-cz/q} \ll  \frac{x}{(\log_3(x))^2}\exp\bigg(\frac{-\log_3(x)}{2^{k+1}}\bigg).$$

Now we sum over all such $k$. Note that for all $t > 0$ that $e^{-t} < \frac{1}{t}$. We see that
\begin{align*}
 \sum_{k=1}^{S} \sum_{q \in I_k} \frac{x}{q \log(q)} e^{-cz/q}
 &\ll \frac{x}{(\log_3(x))^2}\sum_{k = 1 }^{S}\exp\bigg(\frac{-\log_3(x)}{2^{k+1}}\bigg) \\
 &\ll\frac{x}{(\log_3(x))^2}\sum_{k = 1 }^{S}\frac{2^{k+1}}{\log_3(x)}\\ & \ll \frac{x}{(\log_3(x))^3} 2^{S+2}\ll \frac{x}{(\log_3(x))^2}.
\end{align*}
\noindent Thus, the number of $n \leq x$ where $n \in (N \setminus C)$ with a prime factor $p \leq z$ is $ O\bigg( \frac{x}{(\log_3(x))^2}\bigg)$.

This leaves us with the remaining case: $n \in (N\setminus C)$ with $n$ divisible by primes $p, q > z$ with $q\equiv 1 \mod p$. Write $n=pqk$. The number of $n\leq x$ that are products of two distinct primes is $O(x\log_2(x)/\log(x))$ (\cite[Theorem 437, p. 491]{Wright}), so we can assume $k > 1$. In fact, as $n$ has no prime factors less than $z$, we can assume $k \geq z$. Then the number of such $n$ is at most

$$ \sum_{z \leq p\leq x} \sum_{\substack{q \leq x\\ q \equiv 1 \mod p}}\hspace{.1in} \sum_{\substack{z \leq k \leq x/pq \\ r \mid k  \rightarrow r \geq z}} 1 \ll \sum_{z \leq p\leq x} \sum_{\substack{q \leq x\\ q \equiv 1 \mod p}} \frac{x}{pq\log(z)}$$

\noindent by Brun's sieve and Mertens' theorem. Rewriting, the last expression is
\begin{align*}
\frac{x}{\log_3(x)} \sum_{z \leq p\leq x} \frac{1}{p} \sum_{\substack{q \leq z\\ q \equiv 1 \mod p}} \frac{1}{q} & \ll \frac{x}{\log_3(x)} \sum_{z \leq p\leq x} \frac{1}{p} \frac{\log_2(x)}{p}\\ & \ll \frac{x \log_2(x)}{\log_3(x)} \sum_{p \geq z} \frac{1}{p^2} \\ &\ll \frac{x\log_2(x)}{\log_2(x)(\log_3(x)^2)} = \frac{x}{(\log_3(x))^2}
\end{align*}

\noindent using the Brun-Titchmarsh inequality to bound the sum on $q$, and partial summation for the sum on $p$. Putting it all together we see that $(N \setminus C)(x) = O \bigg(\frac{x}{(\log_3(x))^2}\bigg)$. \qedhere

\end{proof}

The following lemma is the only place where our assumption \eqref{eq:assume} on the count of Wieferich primes is used.

\begin{lemma}\label{cn} Assume \eqref{eq:assume}. Let $(C \setminus N)(x) := \#\{ n \leq x : n \text{ is cyclic, and Wieferich} \}$. We have that

$$(C\setminus N)(x) = O\bigg( \frac{x}{(\log_3(x))^2} \bigg).$$
    
\end{lemma}

\begin{proof}
Suppose $n$ is counted by $(C \setminus N)(x)$. Then $n$ is both squarefree and there exists $q \mid n$ such that $q \mid (2^{\ell(n)}-1)/n$. Write $n = q \cdot p_1  \cdots p_k$, with $(n,\varphi(n)) =1$. Writing $\ell(n) = j\cdot\ell(q)$, we see that
$$q \mid \bigg(\frac{2^{\ell(q)}-1}q \bigg)\bigg(2^{\ell(q)(j-1)}+2^{\ell(q)(j-2)} +\dots + 2^{\ell(q)} \bigg),$$ 

\noindent where the rightmost term is congruent to $j \mod q$. If $q$ were to divide that term, then
$$q\mid j \mid \ell(n/q) =  \text{lcm}(\ell(p_1),\ell(p_2),\cdots, \ell(p_k)).$$ 

\noindent Hence, for some $1 \leq i \leq k$, it must be true that
$$q \mid \ell(p_i) \mid \varphi(p_i) \mid \varphi(n)$$

\noindent contradicting $\gcd(n,\varphi(n)) =1$. Therefore, $q$ must divide the left term, and so $q$ is Wieferich.

We will now count the number of $n\leq x$ divisible by a Wieferich prime. We split this collection of $n$ into two sets according to whether or not $n$ has a prime factor $p \leq z$. Here, $z$ is as in \eqref{eq:z}. To handle $n$ with no prime factor at most $z$, notice that
$$\# \{ n \leq x : n = pm \text{ where $p>z$ is a Wieferich prime and $r\mid m \implies r>z$} \}$$
$$ \leq \sum_{m \leq \sqrt{x}}W\bigg(\frac{x}{m}\bigg)+\sum_{\substack{z<p\leq \sqrt{x}\\ p\text{ Wieferich}}}\#\{m<x/p : r\mid m \implies r>z\}.$$

Using our assumed bound \eqref{eq:assume} of $W(y) = O(y/\log^\alpha(y))$ for some $\alpha>2$, we get that the number of these $n$ is
\begin{align*}
& \ll \sum_{m\leq \sqrt{x}}\frac{x/m}{\log^\alpha(x)} + \sum_{\substack{z<p\leq \sqrt{x}\\ p\text{ Wieferich}}}\frac{x}{p}\prod_{\substack{q \text{ prime} \\ q<z}} \bigg( 1 - \frac{1}{q}\bigg)  \\ &\ll \frac{x}{\log^\alpha(x)} \sum_{m\leq \sqrt{x}} \frac{1}{m} + \frac{x}{\log_3(x)} \sum_{\substack{p>z\\ p\text{ Wieferich}}}\frac{1}{p} \\ & \ll \frac{x}{\log^{\alpha-1}(x)} + \frac{x}{\log_3(x)}\frac{1}{\log_3^{\alpha-1}(x)}
\end{align*}
\noindent using partial summation over the Wieferich primes. Thus, the number of such $n\leq x$ is $O\bigg(\frac{x}{(\log_3(x))^2} \bigg)$.

The case where $n$ is divisible by a prime $p\leq z$, is handled in the proof of \eqref{eq:cyclic} in \cite{Vaughan}. It is shown in \cite{Vaughan} on page 388 that 
$$\# \{ n \leq x : (n,\varphi(n)) = 1 \text{ and } \exists p \text{ prime with }  p \mid n \text{ and } p\leq z \} \ll \frac{x}{(\log_3(x))^2}.$$

Hence, $(C \setminus N)(x) = O(x/(\log_3(x))^2)$. \qedhere

\end{proof}

\begin{proof}[Proof of Theorem \ref{main}]Combining Lemmas \ref{nc} and \ref{cn}, we see that
$$S(x) = (C\setminus N)(x) + (N \setminus C)(x) = O\bigg( \frac{x}{(\log_3(x))^2} \bigg)$$

\noindent which concludes the proof of Theorem \ref{main}.  \qedhere

\end{proof}

\begin{proof}[Proof of Corollary \ref{asyms}]

Combining Lemma \ref{nc} and \eqref{eq:cyclic}, we get the following unconditional upperbound:
\begin{equation}\label{eq:Nupper}
    N(x) \leq C(x) + (N\setminus C)(x) \sim \frac{e^{-\gamma}x}{\log_3(x)}.
\end{equation}

\noindent Moreover, with the assumption of \eqref{eq:assume}, applying Theorem \ref{main} gives us a matching lowerbound,
$$N(x) \geq C(x) - (C\setminus N)(x) \sim  \frac{e^{-\gamma}x}{\log_3(x)},$$

\noindent concluding the proof of Corollary \ref{asyms}. \qedhere

\end{proof}

Recall that $U = \{  n \leq x : n \text{ is not $\lambda$-Wieferich}\}$ and $U(x) = \#\{n\leq x : n \in U \}$. As $U \subseteq N$,
\begin{equation}\label{eq:Uupper}
    U(x) \leq N(x).
\end{equation}

\begin{lemma}\label{diff} $(N \setminus U) \subset (N \setminus C)$.

\end{lemma}

\begin{proof}

Let $n \in (N \setminus U)$. As $n \notin U$, there exists a prime $p \mid n$ such that $p^2 \mid 2^{\lambda(n)}-1$.  By Lemma \ref{baby}, as $n$ is non-Wieferich, $p$ must also be non-Wieferich, and thus $p^2 \nmid 2^{\ell(p)}-1$. As $\ell(p) \mid \lambda(n)$, write $\lambda(n) = \ell(p)k$ for some $k$. Then we see that $p$ must divide the following:
$$\frac{2^{\lambda(n)}-1}{2^{\ell(p)}-1} = 2^{\ell(p)(k-1)}+2^{\ell(p)(k-2)}+ \cdots + 2^{\ell(p)} + 1.$$

\noindent But each term on the right is $1 \mod p$, and there are $k$ terms, so the entire quantity is $k \mod p$. Thus $p \mid k \mid \lambda(n) \mid \varphi(n)$. Therefore, $n \notin C$.

\end{proof}

\begin{proof}[Proof of Theorem \ref{improved}] By Lemma \ref{diff}, we get
$$U(x) = N(x) - (N\setminus U)(x) \geq N(x) - (N \setminus C)(x).$$

\noindent Combined with Lemma \ref{nc} and \eqref{eq:Uupper}, we see
$$ U(x) = N(x) + O\bigg(\frac{x}{(\log_3(x))^2} \bigg)$$
    
\noindent unconditionally. Moreover, with the assumption of \eqref{eq:assume}, applying Corollary \ref{asyms} gives that
$$U(x) \sim \frac{e^{-\gamma}x}{\log_3(x)}$$

\noindent concluding the proof of Theorem \ref{improved}. \qedhere

\end{proof}

\bibliographystyle{alpha}
\bibliography{mainbib}

\end{document}